\documentclass[10pt]{article}

\usepackage[colorlinks,linkcolor=blue,citecolor=blue]{hyperref}
\usepackage{amsthm}
\usepackage{multirow}
\usepackage{marginnote}
\usepackage{a4wide}
\usepackage{amssymb}
\usepackage{amsfonts}
\usepackage{amsmath}
\usepackage{mathrsfs}
\usepackage{tikz}
\usepackage{tikz-cd}
\usetikzlibrary{arrows,matrix}
\usetikzlibrary{positioning}
\usepackage{mdframed} 
\usepackage{lipsum} 
\usepackage{extarrows} 
\usepackage{color}
\usepackage{enumitem}
\usepackage{tocloft} 
\usepackage{titlesec} 
\usepackage[all,cmtip]{xy}
\input xy
\xyoption{arrow} \xyoption{matrix}

\DeclareMathAlphabet{\mathpzc}{OT1}{pzc}{m}{it}
\DeclareMathAlphabet\EuScript{U}{eus}{m}{n}

\date{}

\newtheorem{proposition}{Proposition}[section]
\newtheorem{theorem}[proposition]{Theorem}
\newtheorem{lemma}[proposition]{Lemma}

\def\der{\partial }

\def\nFM0{{\nu }_{F,M_0}}
\def\nFN0{{\nu }_{F,N_0}}
\def\nGN0{{\nu }_{G,N_0}}

\def\N0{ {\bf N}_0 }

\def\ra{\rightarrow}

\def\Xpm{X^{\pm }}

\def\s{\sigma}

\def\l1{{\lambda}_1}

\def\a{\alpha}
\def\a0{ {\alpha }_0}
\def\a1{ {\alpha }_1}

\def\l{\lambda}

\def\nFGM0{{\nu }_{F,G,M_0}}

\def\nFN0{{\nu}_{F,N_0}}

\def\sm{{\sigma}^m}

\def\sm1{{\sigma}^{-1}}

\def\smtp1{{\sigma}^{-t+1}}

\def\S1{S^{-1}}

\def\Xpm1{X^{\pm 1}_1}

\def\sPM1{{\sigma }^{\pm 1}}
\def\sMP1{{\sigma }^{\mp 1 }}

\def\b{\beta}

\def\di{{\rm d.ind}}

\def\L{\Lambda}

\def\Ytm1{Y^{t-1}}
\def\Yim1{Y^{i-1}}

\def\CK{{\cal K}}

\def\ass{{\rm ass}}

\def\ker{ {\rm ker } }

\def\SL2Z{ {\rm SL}_2({\bf Z}) }

\def\th{ \theta }

\def\Gp1{ G^{1 , 1 } }
\def\P11{ P^{-1 , 1 } }
\def\Pp1{ P^{1 , 1 } }

\def\th{\theta}

\def\nCLsr{{}^\nu\kern-2pt {\cal L}^{\sigma , \rho  }}
\def\nP{{}^\nu \kern-2pt P}
\def\nL{{}^\nu\kern-2pt L}
\def\nLL{{}^\nu\kern-2pt \Lambda}
\def\nPsr{{}^\nu\kern-2pt P^{\sigma , \rho  }}
\def\nLsr{{}^\nu\kern-2pt L^{\sigma , \rho  }}
\def\nuCL{{}^\nu\kern-2pt  {\cal L}}
\def\nCLsr{{}^\nu\kern-2pt {\cal L}^{\sigma , \rho  }}
\def\nCL1m{{}^\nu\kern-2pt {\cal L}^{-1 , 1  }}
\def\x1nu{x^\frac{1}{\nu}}
\def\xm1nu{x^{-\frac{1}{\nu}}}

\def\ra{\rightarrow }

\def\CB{{\cal B}}

\def\CC{ {\cal C}}

\def\nAM0{{\nu }_{{\cal A},M_0}}
\def\nAN0{{\nu }_{{\cal A},N_0}}

\def\End{ {\rm End }}

\def\bR{\overline{R}}

\def\bx{\overline{x}}

\def\ga{\mathfrak{a}}

\def\SL{{\rm SL}}

\def\di!{\frac{\der^i}{i!}}
\def\dik!{\frac{\der^k_i}{k!}}

\def\gl{\mathfrak{l}}

\def\N{\mathbb{N}}

\def\0{\overline{0}}
\def\1{\overline{1}}

\def\Ln1{\L_{n,\overline{1}}}

\def\a1{a_{\overline{1}}}

\def\S{\Sigma}

\def\vn1{\overrightarrow{n-1}}

\def\gl{{\rm gl}}
\def\sl{{\rm sl}}

\def\mJ{\mathbb{J}}
\def\mI{\mathbb{I}}

\def\K1{{\rm K}_1}

\def\hmI1{\widehat{\mI_1}}
\def\tmI1{\widetilde{\mI_1}}
\def\tmJ1{\widetilde{\mJ_1}}
\def\hB1{\widehat{B_1}}
\def\hCB1{\widehat{\CB_1}}

\def\Den{{\rm Den}}

\def\Ore{{\rm Ore}}

\def\Den{{\rm Den}}

\def\br{\overline{r}}

\def\ga{\mathfrak{a}}

\def\udim{{\rm udim}}

\def \S{\mathcal{S}}

\def\sl2{\mathfrak{sl}_2}
\def\gl2{\mathfrak{gl}_2}

\def\b1{\overline{1}}

\makeatletter
\newenvironment{proof*}[1][\proofname]{\par
  \pushQED{\qed}%
  \normalfont \partopsep=\z@skip \topsep=\z@skip
  \trivlist
  \item[\hskip\labelsep
        \itshape
    #1\@addpunct{.}]\ignorespaces
}{%
  \popQED\endtrivlist\@endpefalse
}
\makeatother

\begin{document}

\author{V. V. \  Bavula 
}

\title{Affirmative answer to the Question of Leroy and Matczuk on injectivity of endomorphisms  of semiprime left Noetherian rings with large images}

\maketitle

\begin{abstract}
 The class of semiprime left Goldie rings is a huge class of rings that contains many  large subclasses of rings --  
semiprime left Noetherian rings,  semiprime rings with Krull dimension, rings of differential operators on  affine algebraic varieties and universal enveloping algebras  of finite dimensional Lie algebras to name a few.  In  the paper, `Ring endomorphisms with large images,' {\em  Glasg. Math. J.} {\bf  55} (2013), no. 2, 381--390,  A. Leroy and J.  Matczuk posed the following  question:
  
   {\em If a ring endomorphism of a
semiprime left Noetherian ring has a large image, must it be injective?}

    The aim of the paper is to give an affirmative answer to the Question of Leroy and Matczuk and to prove the following more general results. \\

{\bf Theorem. (Dichotomy)} {\em Each endomorphism of a semiprime left Goldie ring with large image  is either  a monomorphism  or otherwise  its  kernel contains a regular element of the ring ($\Leftrightarrow$ its  kernel is  an essential left ideal of the ring). In general,  both cases are non-empty.}\\

{\bf Theorem. } {\em Every  endomorphism with large image  of a semiprime ring with Krull dimension  is a monomorphism.}\\

{\bf Theorem. (Positive answer to the Question of Leroy and Matczuk)} {\em Every  endomorphism with large image  of a semiprime left Noetherian ring  is a monomorphism.}\\




{ \small \tableofcontents}
\end{abstract}


\section{Introduction}

In the paper, module means a left module,  $R$ is a unital ring, $\End (R)$ is the semigroup of its  ring endomorphisms,  $\CC_R$ is the set of regular elements (non-zero-divisors)  of $R$ and $Q_l(R):=\CC_R^{-1}R$  is its (classical) left quotient ring (the classical ring of left fractions). 

A left ideal of a ring  is an {\bf essential left ideal} if it has nonzero intersection with every nonzero left ideal of the ring. An endomorphism of a ring  has {\bf large image} if the image contains a left essential ideal of the ring. 

Recall that a ring $R$ is a {\bf left Goldie ring} if it has  {\bf finite left uniform dimension} (i.e. the ring $R$ does not contains an infinite direct sum of nonzero  left ideals) and satisfies the ascending chain condition  (the a.c.c.) on left annihilators. {\bf  Goldie's Theorem} states that {\em a ring $R$ is a semiprime left Goldie ring iff its left quotient ring $Q_l(R)$ is a semisimple Artinian ring}. The key fact in the proof of Goldie's Theorem is that {\em a left ideal of a semiprime left Goldie ring is essential iff it contains a regular element of the ring}.

In  \cite{Leroy-Mat-2013},  Leroy and   Matczuk  study ring endomorphisms of a ring $R$ which have
large images. They prove a number of results assuming different hypotheses. In particular, they proved that:
 
(1) Let $\s$  be an endomorphism with  large image of a prime ring $R$ which has left Krull dimension. Then $\s$ is injective. 

(2) Let $L$ be
a left essential ideal of a semiprime left Noetherian ring $R$. If $\s$ is a ring endomorphism
of $R$ such that $\s (L) = L$, then $\s$ is an automorphism of $R$.

In  \cite{Leroy-Mat-2013},   Leroy and   Matczuk posed the following two  question:
  
 {\em If a ring endomorphism of a
semiprime left Noetherian ring has a large image, must it be injective?}


The aim of the paper is to give an affirmative answer to the Question of Leroy and Matczuk and to prove the following three theorems.

\begin{theorem}\label{13Sep24}
{\bf (Dichotomy)} Each endomorphism of a semiprime left Goldie ring with large image  is either a monomorphism  or otherwise  its  kernel contains a regular element of the ring ($\Leftrightarrow$ its  kernel is  an essential left ideal of the ring). In general,  both cases are non-empty.
\end{theorem}

 There are  proper containments of the following  classes of rings:
\begin{eqnarray*}
\{\rm semiprime\; left\; Goldie\; rings\} &\supset& \{\rm semiprime\;  rings\; with\; Krull\; dimension\}\\
&\supset&  \{\rm semiprime\; left\; Noetherian\; rings\},
\end{eqnarray*}
This follows from the following facts.
 A semiprime ring with Krull dimension is a left  order in a semisimple artinian ring, \cite[Theorem, p. 519]{Gordon-Robson-1973}, i.e. a semiprime ring with Krull dimension is a semiprime left Goldie ring.  Any Noetherian module has Krull dimension,  \cite{Gabriel-1962}. 
 Semiprime ring with Krull dimension need not be left Noetherian, \cite[Examples 10.7 and 10.8]{Gordon-Robson-KrullDim}. On the other hand,  a semiprime left Goldie ring need not have a Krull
dimension (e.g. the polynomial ring in infinitely many commuting indeterminates over a field).
 Semiprime left Goldie rings are precisely
those semiprime rings with finite uniform dimension and enough monoform left ideals, \cite{Gordon-Robson-KrullDim}.

\begin{theorem}\label{A13Sep24}
Every  endomorphism with large image  of a semiprime ring with Krull dimension  is a monomorphism.
\end{theorem}
 
 Proofs of  Theorem \ref{13Sep24} and  Theorem \ref{A13Sep24} are  given in Section  \ref{ENDOM-SLGOLDIE}.
  
  Theorem \ref{B13Sep24}  gives  an affirmative answer to the Question of Leroy-Matczuk.

 \begin{theorem}\label{B13Sep24}
Every  endomorphism with large image  of a semiprime left Noetherian ring  is a monomorphism.
\end{theorem}

 \begin{proof} Theorem \ref{B13Sep24} is a particular case of Theorem \ref{A13Sep24}.
 \end{proof}

The example below shows that in general both cases in Theorem \ref{13Sep24} are non-empty even in the case of commutative domains. Let $P_\infty =K[x_0, x_1, \ldots ]$ be a polynomial algebra in countably many variables over a field $K$ and an endomorphism  $\s \in \End_K(P_\infty)$ is given by the rule: 
$$
\s (x_0)=0\;\; {\rm and}\;\;\s(x_i)=x_{i-1}\;\; {\rm for\; all}\;\;i\geq 1.
$$
 The algebra $P_\infty$ is a commutative non-Noetherian domain. It is also a prime/semiprime left Goldie ring since the ring  $Q_l(P_\infty )$ is the field of rational functions  $Q_\infty (K):=K(x_0, x_1, \ldots )$ in countably many variables. The endomorphism $\s$ is an {\em epimorphism} with nonzero kernel $(x_0)$ which is an essential ideal of $P_\infty$. In particular, the endomorphism  $\s$  is an endomorphism  of $P_\infty$ with large images which is not  a monomorphism.

So, in general, the statement that `{\em Endomorphisms of prime/semiprime left Goldie rings with large images  are injective}' is wrong.

 Furthermore, an example of Leroy and Matczuk  \cite[Example 2.12]{Leroy-Mat-2013} shows that the endomorphism $\s$  in Theorem \ref{13Sep24} is a monomorphism which is {\em not} an automorphism.

Another example shows that the semiprimeness condition in Theorem \ref{13Sep24} cannot be dropped. Let $R=K[x]/(x^{n+1})$ where $K$ is a field and $n\geq 2$. The ring $R$ is a left Goldie ring but not a semiprime ring since the ideal $(\bx)$ is a nonzero nilpotent ideal  where $\bx = x+(x^{n+1})$. 
 The algebra $R$ is uniserial, i.e. the strictly asscending chain of ideals of $R$,  $0\subset (\bx^n)\subset \cdots \subset (\bx)$, contains all the ideals of $R$. The endomorhism $\s \in \End (R)$, where $\s (\bx)=\bx^n$, has large image since $\s (R) = K\oplus K\bx^n$. The endomorphism  $\s$ is not a monomorphism since $\ker (\s) =(\bx^2)$. The kernel $\ker (\s) =(\bx^2)$, which is a nilpotent ideal of $R$,  contains no regular elements of the ring $R$.


\section{Endomorphisms of semiprime left Goldie rings with large image} \label{ENDOM-SLGOLDIE} 
In the section, proofs of Theorem \ref{13Sep24} and Theorem \ref{A13Sep24} are given. \\

{\bf Left denominator sets, localizations, the left quotient ring of a ring and Goldie's Theorem.} 
  A subset $S$ of $R$ is called a {\em multiplicative set} if  $SS\subseteq S$, $1\in S$ and $0\not\in S$. A 
multiplicative subset $S$ of $R$   is called  a {\em
left Ore set} if it satisfies the {\em left Ore condition}: for
each $r\in R$ and
 $s\in S$, $$ Sr\bigcap Rs\neq \emptyset .$$
Let $\Ore_l(R)$ be the set of all left Ore sets of $R$.
  For  $S\in \Ore_l(R)$, $\ass_l (S) :=\{ r\in
R\, | \, sr=0 \;\; {\rm for\;  some}\;\; s\in S\}$  is an ideal of
the ring $R$. 


A left Ore set $S$ is called a {\em left denominator set} of the
ring $R$ if $rs=0$ for some elements $ r\in R$ and $s\in S$ implies
$tr=0$ for some element $t\in S$, i.e., $r\in \ass_l (S)$. Let
$\Den_l(R)$ (resp., $\Den_l(R, \ga )$) be the set of all left denominator sets of $R$ (resp., such that $\ass_l(S)=\ga$). For
$S\in \Den_l(R)$, let $$S^{-1}R=\{ s^{-1}r\, | \, s\in S, r\in R\}$$
be the {\em left localization} of the ring $R$ at $S$ (the {\em
left quotient ring} of $R$ at $S$). 
 In a similar way, right Ore and right denominator sets are defined. 
Let $\Ore_r(R)$ and 
$\Den_r(R)$ be the set of all right  Ore and  right   denominator sets of $R$, respectively.  For $S\in \Ore_r(R)$, the set  $\ass_r(S):=\{ r\in R\, | \, rs=0$ for some $s\in S\}$ is an ideal of $R$. For
$S\in \Den_r(R)$,  $$RS^{-1}=\{ rs^{-1}\, | \, s\in S, r\in R\}$$   is the {\em right localization} of the ring $R$ at $S$. 

Given ring homomorphisms $\nu_A: R\ra A$ and $\nu_B :R\ra B$. A ring homomorphism $f:A\ra B$ is called an $R$-{\em homomorphism} if $\nu_B= f\nu_A$.  A left and right Ore set is called an {\em Ore set}.  Let $\Ore (R)$ and 
$\Den (R)$ be the set of all   Ore and    denominator sets of $R$, respectively. For
$S\in \Den (R)$, $$S^{-1}R\simeq RS^{-1}$$ (an $R$-isomorphism)
 is  the {\em  localization} of the ring $R$ at $S$, and $\ass (S):=\ass_l(S) = \ass_r(S)$. \\

  For each element $r\in R$, let $r\cdot : R\ra R$, $x\mapsto rx$ and $\cdot r : R\ra R$, $x\mapsto xr$. The sets $${}'\mathcal{C}_R := \{ r\in R\, | \, \ker (\cdot r)=0\}\;\; {\rm  and}\;\;\mathcal{C}_R' := \{ r\in R\, | \, \ker (r\cdot )=0\}$$ are called the {\em sets of left and right regular elements} of $R$, respectively.  Their intersection $$\mathcal{C}_R={}'\mathcal{C}_R \cap \mathcal{C}_R'$$ is the  {\em set of regular elements}  of $R$. The rings $$Q_l(R):= \mathcal{C}_R^{-1}R\;\; {\rm  and}\;\;Q_r(R):= R\mathcal{C}_R^{-1}$$ are called the {\em (classical) left and right quotient rings} of $R$, respectively. If both rings exist then they are isomorphic and the ring
 $$Q(R):= Q_l(R)\simeq  Q_r(R)$$
 is called the {\em (classical) quotient ring} of $R$.   
 Goldie's Theorem states that the ring $Q_l(R)$ is  a semisimple Artinian ring iff the ring $R$ is a semiprime left Goldie ring (the ring $R$ is called a {\em left Goldie ring} if   $\udim (R)<\infty$ and the ring $R$ satisfies the a.c.c. on left annihilators where $\udim$ stands for the {\em left uniform dimension}). In a similar way a {\em right Goldie ring} is defined. A left and right Goldie ring is called a {\em Goldie ring}.

  Proposition \ref{A19Sep24} is a useful result that is used in the proof of Theorem \ref{A11Sep24}. Theorem \ref{A11Sep24} is used in the proof of Theorem \ref{11Sep24}.

 \begin{proposition}\label{A19Sep24}
Let $R$ be a ring, $S\in \Den_l(R,\ga)$, $\th : R\ra S^{-1}R$, $r\mapsto \frac{r}{1}$  and $T$ be a multiplicative subset of $R$ such that $S\subseteq T$ and $\th (T)\subseteq (S^{-1}R)^\times$. Then:
\begin{enumerate}

\item $T\in \Den_l(R, \ga)$ and $T^{-1}R\simeq S^{-1}R$.


\item Suppose, in addition,   that $S\in \Den_l(R,0)$,  the ring $S^{-1}R$ is a semisimple Artininan ring and $T=\CC_R$. Then $\CC_R\in \Den_l(R, 0)$ and $Q_l(R)\simeq S^{-1}R$.

\end{enumerate} 
\end{proposition}

\begin{proof} 1. (i) $\ass_l(T)=\ga$: By the assumption $S\subseteq T$. Therefore, $\ga=\ass_l(S)\subseteq \ass_l(T)$. The reverse inclusion follows from the fact that  $\th (T)\subseteq (S^{-1}R)^\times$.

(ii) $T\in \Ore_l(R)$: We have to show that the multiplicative set $T$ satisfies the left  Ore condition: For any elements $t\in T$ and $r\in R$, there are elements  $t'\in T$ and $r'\in R$ such that 
$$t'r=r't.$$
 In the ring $S^{-1}R$, we have the equality 
 $ rt^{-1}= s^{-1}r_1$ for some elements $s\in S$ and $r_1\in R$. Hence, $sr-r_1t\in \ga$, and so $s'(sr-r_1t)=0$ for some element $s'\in S$, i.e.  $s'sr=s'r_1t$. Now, it suffices to take $t'=s's\in S\subseteq T$ and $r'=s'r_1\in R$.

(iii) $T\in \Den_l(R,\ga)$: In view of the statements (i) and (ii), we have to show that $\ass_r (T)\subseteq \ga$. Suppose that $rt=0$ for some elements $r\in R$ and $t\in T$. Then $S^{-1}R\ni \frac{r}{1}\frac{t}{1}=0$, and so $\frac{r}{1}=0$ since $\frac{t}{1}\in (S^{-1}R)^\times $. Therefore, $r\in \ga$ and $\ass_r (T)\subseteq \ga$.

(iv) $T^{-1}R\simeq S^{-1}R$: Since $\th (T)\subseteq (S^{-1}R)^\times$ and $T\in \Den_l(R, \ga)$ (the statement (iii)), the map  
$$\eta: T^{-1}R\ra S^{-1}R, \;\;t^{-1}r\mapsto t^{-1}r$$
is a well-defined homomorphism. The homomorphism $\eta$ is an epimorphism (since $S\subseteq T$) and a monomorphism (since  $\th (T)\subseteq (S^{-1}R)^\times$ and $S\in \Den_l(R,\ga)$). Therefore, $\eta$ is an isomorphism. 

2. Since $S\in \Den_l(R,0)$, the ring $R$ is a subring of the ring $S^{-1}R$ and an essential left $R$-submodule of $S^{-1}R$. The ring $S^{-1}R$ is a semisimple Artininan ring. Therefore, 
$$\Big(S^{-1}R \Big)^\times =\CC_{S^{-1}R}={}'\CC_{S^{-1}R}=\CC_{S^{-1}R}'.$$
 Then  $T=\CC_R\subseteq {}'\CC_{S^{-1}R}=\Big(S^{-1}R \Big)^\times$, and so  $\CC_R\in \Den_l(R, 0)$ and $Q_l(R)\simeq S^{-1}R$, by statement 1.
\end{proof}

{\bf Ring endomorphisms that respect the set of regular elements.} For a ring $R$, 
$$\End(R, \CC_R):=\{ \s \in \End (R)\, | \, \s (\CC_R)\subseteq \CC_R\}$$ is the set of endomorphisms of $R$ that {\em respect the set of regular elements} of the ring  $R$. Suppose that $\CC_R\in \Ore_l(R)$ or equivalently $\CC_R\in \Den_l(R,0)$.  By the universal property of left localization, each endomorphism $\s\in \End(R, \CC_R)$ can be uniquely extended to an endomorphism $\s\in \End(Q_l(R))$ by the rule:
\begin{equation}\label{sQlR}
\s :Q_l(R)\ra Q_l(R), \;\; cr\mapsto   \s (c)^{-1}\s (r)\;\; {\rm where}\;\; c\in \CC_R \;\; {\rm and}\;\;r\in R.
\end{equation}

For a ring $R$ and its module $M$, we denote by $l_R(M)$ its {\em length}. 

Theorem \ref{A11Sep24} is about properties of the images of endomorphisms of  semiprime left Goldie rings that respect the set of  regular elements.

\begin{theorem}\label{A11Sep24}
Let $R$ be a semiprime left Goldie ring, $Q_l(R)=\prod_{i=1}^s Q_i $ where the rings $Q_i$ are simple Artinian rings  and $\s\in \End(R, \CC_R)$. Then:
\begin{enumerate}

\item The ring  $\s (R)$ is  a semiprime left Goldie ring. 

\item The map $\s : Q_l(R)\ra Q_l(\s (R))$, $c^{-1}r\mapsto \s (c)^{-1}\s (r)$ is a split epimorphism where $c\in \CC_R$ and $r\in R$, that is $Q_l(R)\simeq \ker_{Q_l(R)}(\s)\times Q_l(\s (R))$ and $Q_l(\s (R))= \s (Q_l(R))$.


\item $\udim (R)= \udim (Q_l(R))=\udim(Q_l(\s (R)))+l_{Q_l(R)} (\ker_{Q_l(R)} (\s))=
\udim(\s (R))+\udim (\ker_R (\s))$. 

\end{enumerate}
\end{theorem}

\begin{proof} 1--2. Since $\s\in \End(R, \CC_R)$, the endomorphism $\s$ can be uniquely extended to the endomorphism (see (\ref{sQlR})):
 $$\s :Q_l(R)\ra Q_l(R), \;\; cr\mapsto   \s (c)^{-1}\s (r)\;\; {\rm where}\;\; c\in \CC_R \;\; {\rm and}\;\;r\in R.$$

(i)  {\em Up to order, $\ker_{Q_l(R)}(\s)=Q_1\times\cdots \times Q_m\times \{ 0\} \times \cdots\times \{ 0\}=\CC_R^{-1}\ker_R(\s)$ for some $m$ such that} $0\leq m<n$ ($m=0$ iff $\ker_R(\s)=0$): Since  $Q_l(R)=\prod_{i=1}^s Q_i $ where the rings $Q_i$ are simple Artinian rings and 
$\ker_{Q_l(R)} (\s)$ is an ideal of $Q_l(R)$,  
$$\ker_{Q_l(R)}(\s)=Q_1\times\cdots \times Q_m\times \{ 0\} \times \cdots\times \{ 0\}\supseteq \CC_R^{-1}\ker_R(\s)$$ for some $m$ such that $0\leq m<n$, up to order. 

Suppose that $c^{-1}r\in \ker_{Q_l(R)}(\s)$ where $c\in \CC_R$ and $r\in R$. Then $0=\s(c^{-1}r)=\s(c)^{-1}\s (r)$, and so $\s (r)=0$ (since $\s\in \End(R, \CC_R)$), i.e. $r\in \ker_R (\s)$. Therefore, $$
\ker_{Q_l(R)}(\s)=\CC_R^{-1}\ker_R(\s).
$$
(ii) $\s (Q_l(R))=\{ \s (c)^{-1}\s (r)\, | \, c\in \CC_R, r\in R\}\simeq Q_{m+1}\times\cdots \times Q_s$ {\em is a semisimple Artinian ring and the map $ \s :Q_l(R)\ra \s(Q_l(R))$ is a split epimorphism and}  $Q_l(R)\simeq \ker_{Q_l(R)}(\s)\times \s (Q_l(R))$: Clearly, $\s (Q_l(R))=\{ \s (c)^{-1}\s (r)\, | \, c\in \CC_R, r\in R\}$. By the statement (i), 
$$
\s (Q_l(R))\simeq Q_l(R)/\ker_{Q_l(R)}(\s)= \bigg(\prod_{i=1}^s Q_i\bigg)/\Big( Q_1\times\cdots \times Q_m\times \{ 0\} \times \cdots\times \{ 0\}\Big)\simeq Q_{m+1}\times\cdots \times Q_s
$$
is a semisimple Artinian ring. Clearly, the map $ \s :Q_l(R)\ra \s(Q_l(R))$ is a split epimorphism and $$Q_l(R)\simeq \ker_{Q_l(R)}(\s)\times \s (Q_l(R)).$$
 
(iii) $\s (\CC_R)\subseteq \CC_{\s (R)}$: The  statement (iii) follows from the inclusions $\s (\CC_R)\subseteq \CC_R$ and $\s (R)\subseteq R$. 

(iv) $\s (\CC_R) \in \Den_l(\s (R), 0)$: Since $\CC_R\in \Ore_l(R)$ and  $0\not\in \s (\CC_R)$ (by the statement (iii)), we have that 
$$\s (\CC_R)\in \Ore_l(\s (R)).$$ 
Since  $\s (\CC_R)\subseteq \CC_{\s (R)}$ (the statement (iii)), $\s (\CC_R) \in \Den_l(\s (R), 0)$. 

(v) {\em The map $\iota: \s(\CC_R)^{-1}\s (R)\ra \s(Q_l(R))$, 
$\s (c)^{-1}\s (r)\mapsto \s (c)^{-1}\s (r)$ is an isomorphism where $c\in \CC_R$ and $r\in R$}: By the definition, the map $\iota$ is a well-defined epimorphism (by the statements (ii) and (iv)). Since 
$$\ker (\iota) =\s (\CC_R)^{-1}\ker_R(\s)=0,$$
 the map $\iota$ is an isomorphism. 

(vi) {\em The ring $\s(\CC_R)^{-1}\s (R)\simeq \s(Q_l(R))\simeq Q_{m+1}\times\cdots \times Q_s$ is a semisimple Artinian ring}: The statement (vi) follows from the statements (ii) and (v). 

(vii) $\CC_{\s (R)}\in \Den_l(\s(R),0)$ {\em and $Q_l(\s (R))\simeq \s (Q_l(R))$  via $\s (c)^{-1}\s (r)\mapsto \s (c)^{-1}\s (r)$  where $c\in \CC_R$ and} $r\in R$:
 Since $\s (\CC_R) \in \Den_l(\s (R), 0)$ (the statement (iv)), $\s (\CC_R)\subseteq \CC_{\s(R)}$ (the statement (iii)) and the ring 
$$\s(\CC_R)^{-1}\s (R)\simeq \s(Q_l(R))\simeq Q_{m+1}\times\cdots \times Q_s$$
 is a semisimple Artinian ring (the statement (vi)), by  Proposition \ref{A19Sep24}.(1) (where $S=\s (\CC_R)$ and $T=\CC_{\s (R)}$),  we have that 
 $$\CC_{\s (R)}\in \Den_l(\s(R),0)\;\; {\rm  and}\;\; Q_l(\s (R))\simeq \s(\CC_R)^{-1}\s (R)\simeq \s (Q_l(R))$$  via $\s (c)^{-1}\s (r)\mapsto \s (c)^{-1}\s (r)$  where $c\in \CC_R$ and $r\in R$.

(viii) {\em The ring $\s (R)$ is a semiprime left Goldie ring}: By Goldie's Theorem, the statement (viii) follows from the statements (vi) and  (vii).

(ix) {\em The map $\s : Q_l(R)\ra Q_l(\s (R))$, $c^{-1}r\mapsto \s (c)^{-1}\s (r)$ is a split epimorphism where $c\in \CC_R$ and $r\in R$, that is $Q_l(R)\simeq \ker_{Q_l(R)}(\s)\times Q_l(\s (R))$ and} $Q_l(\s (R))= \s (Q_l(R))$: By the statement (vii), the map  $\s : Q_l(R)\ra \s ( Q_l(R))= Q_l(\s (R))$ is an epimorphism. Now, the statement (ix)  follows from the statement (ii).

3. The $R$-module $R$ is an essential $R$-submodule of $Q_l(R)$. Therefore, $\udim (R)= \udim (Q_l(R))$. Similarly,   the $\s (R)$-module $\s (R)$ is an essential $\s (R)$-submodule of $Q_l(\s (R))$. Therefore, $\udim (\s(R))= \udim (Q_l(\s (R)))$.
 By statement 2, $Q_l(R)\simeq \ker_{Q_l(R)}(\s)\times Q_l(\s (R))$. Therefore, $\udim (\ker_R (\s))=l_{Q_l(R)} (\ker_{Q_l(R)} (\s))$
   and $$\udim (Q_l(R))=\udim(Q_l(\s (R)))+l_{Q_l(R)} (\ker_{Q_l(R)} (\s)),\;\; {\rm  and\; so}$$
$\udim (R)= \udim (Q_l(R))=\udim(Q_l(\s (R)))+l_{Q_l(R)} (\ker_{Q_l(R)} (\s))=
\udim(\s (R))+\udim (\ker_R (\s))$.
\end{proof}

 Theorem \ref{11Sep24}, which is a difficult part of \cite[Theorem 1.12]{Leroy-Mat-2013}, shows that every large endomorphism of a semiprime left Goldie ring that respects regular elements  is a monomorphism.  Here we give a different and shorter proof.

\begin{theorem}\label{11Sep24}
Let $R$ be a semiprime left Goldie ring   and $\s\in \End(R)$ be a large endomorphism. Suppose that $\s (\CC_R)\subseteq \CC_R$ then $\s$ is a monomorphism. 
\end{theorem}

\begin{proof} Suppose that $\ker (\s)\neq 0$, we seek a contradiction. On the one hand, by Theorem \ref{A11Sep24}.(3), 
$$ \udim(\s (R))=\udim (R)-\udim (\ker_R (\s))<\udim (R).$$

On the other hand,  the endomorphism $\s$ is large, i.e. the ring $\s (R)$ contains an essential left  ideal,    say $I$, of the ring $R$. Therefore, 
$$ \udim(\s (R))\geq  \udim({}_{\s (R)}I)\geq \udim({}_RI)=\udim(R),$$
a contradiction.
\end{proof}

{\bf Large subrings of semiprime left Goldie rings, their regular elements and left quotient rings.} A subring $R'$ of a ring $R$ is called  a {\bf large subring} if the ring $R'$ contains an essential left ideal of $R$.

For a large subring $R'$ of a ring $R$, 
Theorem \ref{20Sep24} gives explicit descriptions of the sets $\CC_{R'}$ and $Q_l(R')$ in terms of the sets $\CC_R$ and $Q_l(R)$. It also shows that all large subrings $R'$ of a ring $R$ are semiprime left Goldie rings.

\begin{theorem}\label{20Sep24}
Let $R$ be a semiprime left Goldie ring, $R'$ be a large subring of $R$ that contains an  essential left ideal  $I$
 of $R$. Then:
\begin{enumerate}

\item $\CC_{R'}=\CC_R\cap R'$.

\item For each element $c\in \CC_R$, there is an element  $c'\in \CC_R$ such that $c'c\in \CC_{R'}$.

\item  $\CC_{R'}\in \Den_l(R',0)\cap \Den_l(R,0)$.

\item $Q_l(R')=\CC_{R'}^{-1}R'=\CC_{R'}^{-1}R=Q_l(R)$.

\item The ring $R'$ is semiprime left Goldie ring.

\item $\udim (R')=\udim (R)$.

\end{enumerate}
\end{theorem}

\begin{proof} 1. The inclusion $\CC_{R'}\supseteq \CC_R\cap R'$ is obvious. It remains to show that the reverse inclusion holds. Given an element $c\in R$,  
 then $c\in \CC_R$ iff $c\in {}'\CC_R$ (since $Q_l(R)$ is a semisimple Artinian ring) iff the map $\cdot c: I\ra R$, $a\mapsto ac$ is injective (since the left ideal $I$ of $R$  essential). Suppose that  $c'\in \CC_{R'}$. Then the map $\cdot c': I\ra R'\subseteq R$, $a\mapsto ac$ is injective. Therefore, $c'\in \CC_R$, and so 
 $\CC_{R'}\subseteq \CC_R\cap R'$.
 
 2. Recall that the ring $R'$ contains an  essential left ideal  $I$ of the ring $R$. Hence, for each element $c\in \CC_R$, there is an element  $c'\in \CC_R$ such that $c'c\in I\subseteq R'$. Clearly, $c'c\in \CC_R\cap R'=\CC_{R'}$, by statement 1.

3. (i) $\CC_{R'}\in \Den_l(R',0))$: We have to show that 
 $\CC_{R'}\in \Ore_l(R')$, i.e. for any pair of elements 
 $c'\in \CC_{R'}$ and $r'\in R'$, there are elements 
$c'_1\in \CC_{R'}$ and $r_1'\in R'$ such that $c_1'r'=r'_1c'$. Since $\CC_{R'}\subseteq \CC_R$ (statement 1) and $\CC_R\in \Den_l(R,0)$, there are elements $c\in \CC_{R}$ and $r\in R$ such that 
$cr'=rc'$. 

Since the left ideal $I$ of $R$ is essential and the ring $R$ is a semiprime left Goldie ring, $I\cap \CC_R\neq \emptyset$. Therefore, $\CC_R^{-1}I=Q_l(R)$. Hence, $c_1r\in I$ for some element $c_1\in \CC_R$. By statement 2, for the element $c_1c\in \CC_R$, there is an element $\xi\in \CC_{R'}$ such that $\xi c_1c\in \CC_{R'}$. Clearly, 
$$ \xi c_1c r'=\xi c_1rc'.$$ 
So, it suffices to take $c_1'=\xi c_1c\in \CC_{R'} $ and $r_1'=\xi \cdot c_1r\in \CC_{R'}I\subseteq I\subseteq R'$.

(ii) $\CC_{R'}\in \Den_l(R,0)$. Since $\CC_{R'}\subseteq \CC_R$ (statement 1), it suffices to show that $\CC_{R'}\in \Ore_l(R)$, i.e. for any pair of elements 
 $c'\in \CC_{R'}$ and $r\in R$, there are elements 
$c'_1\in \CC_{R'}$ and $r_1\in R$ such that $c_1'r=r_1c'$.
 Since $\CC_R\in \Ore_l(R)$, $cr=r'c'$ for  some elements $c\in \CC_R$ and $r'\in R$. By statement 2, $\eta c\in \CC_{R'}$ for some element $\eta \in \CC_{R'}$. Now,
 $$\eta c\cdot r=\eta r'\cdot c'.$$
So, it suffice to take $c_1'=\eta c$ and $r_1=\eta r'$.

4. Recall that $\CC_{R'}\subseteq \CC_R$ (statement 1), $\CC_{R'}\in \Den_l(R,0)$ and $R'\subseteq R$. Therefore,
$$ Q_l(R')=\CC_{R'}^{-1}R'\subseteq \CC_{R'}^{-1}R\subseteq \CC_R^{-1}R=Q_l(R).$$
 By statement 2, we have the reverse inclusion  $ Q_l(R')\supseteq Q_l(R)$. Therefore, $ Q_l(R')= Q_l(R)$.
 
5. By statement 4, the ring $Q_l(R')=Q_l(R)$ is a semisimple Artinian ring. By Goldie's Theorem, the ring $R'$ is a semiprime left Goldie ring.  
 
6. Statement 6 follows from statement 4:  $\udim (R')=\udim (Q_l (R'))=\udim (Q_l (R))=\udim (R)$.
\end{proof}

Suppose that  $R$ is a semiprime left Goldie ring and an endomorphism  $\s \in \End (R)$ has large image. Let $\pi : R\ra \bR :=R/\ker (\s)$, $r\mapsto \br := r+\ker (\s)$. By the Homomorphism Theorem, the map $\overline{\s}: \bR \ra \s (R)$, $ \br\mapsto \overline{\s}(\br )=\s (r)$, is an isomorphism. In particular,
\begin{equation}\label{QlsRR}
\overline{\s}(\CC_{\bR})=\CC_{\s (R)}\;\; {\rm  and}\;\; Q_l(\bR)\stackrel{\overline{\s}}{\simeq} Q_l(\s (R))\simeq Q_l(R),
\end{equation}
by \cite[Proposition 1.9.(3)]{Leroy-Mat-2013} or Theorem \ref{20Sep24}.(4). \\

{\bf Proof of Theorem \ref{13Sep24}.} Let $R$ be a semiprime left Goldie ring and $\s$ be an endomorphism of the ring  $R$ with large image.
 So, there are two cases either $\ker (\s)\cap \CC_R=\emptyset$ or otherwise $\ker (\s)\cap \CC_R\neq \emptyset$. In order to prove Theorem \ref{13Sep24}, it suffices to show that the first case implies that the endomorphism $\s$ is a monomorphism. 
 
 Suppose that $\ker (\s)\cap \CC_R=\emptyset$. We have to show that $\ker (\s)=0$. The ring $Q_l(R)$ is a semisimple Artinian ring. In particular, the ring $Q_l(R)$ a left Noetherian ring. Therefore, for each ideal $\ga$ of $R$, the left ideal  $\CC_R^{-1}\ga$ of $Q_l(R)$ is an {\em ideal} of $Q_l(R)$. 
 By applying the the exact functor $\CC_R^{-1}(-)$ to the short exact sequence of left $R$-modules
 $$0\ra \ker (\s)\ra R\ra \bR=R/\ker (\s)\ra 0,$$
we obtain a short exact sequence of left $R$-modules
 $$0\ra \CC_R^{-1}\ker (\s)\ra \CC_R^{-1}R=Q_l(R)\ra \CC_R^{-1}\bR= \CC_R^{-1}\Big(R/\ker (\s)\Big)\simeq  \CC_R^{-1}R/\CC_R^{-1}\ker (\s) \ra 0.$$
Since the left $Q_l(R)$-module $\CC_R^{-1}\ker (\s)$ is an ideal of the semisimple Artinian algebra $Q_l(R)$, the algebra $Q_l(R)$ is isomorphic to the direct product of  semisimple Artinian rings 
$$Q_l(R)\simeq \CC_R^{-1}\ker (\s)\times \CC_R^{-1}\bR.$$
Clearly, $\CC_R^{-1}\bR \simeq \pi (\CC_R)^{-1}\bR$ and $\pi (\CC_R)\in \Den_l(\bR , 0)$. Since the algebra $\CC_R^{-1}\bR$ is a semisimple Artinian ring,  we must have $\CC_R^{-1}\bR\simeq Q_l(\bR)$. Then
$$ \udim (\CC_R^{-1}\ker (\s))=\udim (Q_l(R))-\udim(\CC_R^{-1}\bR)=\udim (Q_l(R))-\udim(Q_l(R))=0.$$
Hence, $ \udim (\ker (\s))= \udim (\CC_R^{-1}\ker (\s))=0$, i.e. $\ker (\s)=0$. $\Box$\\

{\bf Case: The ring $R$ is a semiprime ring with Krull dimension.} Recall that any semiprime ring with Krull dimension is a semiprime  left Goldie ring. Let $\CK$ be the (left) Krull dimension of a ring or module.

\begin{proposition}\label{A25Sep24}
Let $R$ be a semiprime ring with Krull dimension and an endomorphism  $\s \in \End (R)$ has large image. Then $\CK (\s (R))=\CK (R)$.
\end{proposition}

\begin{proof} (i)  $\CK (\s (R))\leq\CK (R)$:   The ring $\s (R)$ is an epimorphic image of the ring $R$. Hence, $\CK (\s (R))\leq\CK (R)$. 

(ii)  $\CK (\s (R))\geq\CK (R)$:
By the assumption, the endomorphism  $\s \in \End (R)$ has large image. So, the kernel of $\s$ contains an essential left ideal, say $I$, of the ring $R$. 
It follows from the exact sequence of left $R$-modules 
$ 0\ra I \ra R \ra \bR = R/I\ra 0$
that 
$$\CK (R) = \sup \{ \CK (I), \CK ( R/I)\}.$$
The left ideal $I$ of the semiprime left Goldie ring $R$ contains a regular element, say $c\in \CC_R$. Then 
$\CK (R/I)\leq \CK (R/Rc)<\CK (R)$. Therefore,  $\CK (R) =  \CK (I)$. Now, 
$$\CK (R) =  \CK ({}_RI)\leq \CK ({}_{\s(R)}I)\leq \CK (\s (R)),$$
and the statement (ii) follows. 

By the statements (i) and (ii),  $\CK (\s (R))=\CK (R)$.
\end{proof}

\begin{lemma}\label{a15Sep24}
Let $R$ be a semiprime ring with Krull dimension and an endomorphism  $\s \in \End (R)$ has large image. If $\ker (\s)\cap {}'\CC_R\neq \emptyset$ then $\CK (\s (R))\leq \CK (R)-1<\CK (R)$.
\end{lemma}

\begin{proof} Fix an element $c\in \ker (\s)\cap {}'\CC_R$. Then $Rc\subseteq \ker (\s )$, and so
$$ \CK (\s (R))=\CK (R/\ker (\s))\leq \CK (R/Rc)\leq \CK (R)-1<\CK (R).$$
\end{proof}
{\bf Proof of Theorem \ref{A13Sep24}.} Let $R$ be a semiprime  ring with Krull dimension and $\s$ be an endomorphism of the ring  $R$ with large image. In particular, the ring $R$ is a semiprime left Goldie ring. 
 By  Theorem \ref{13Sep24}, the endomorphism $\s$ is  either a monomorphism or otherwise $\ker (\s)\cap \CC_R\neq \emptyset$. So, in order to prove Theorem \ref{A13Sep24} it suffices to show that 
 $$\ker (\s)\cap \CC_R=\emptyset.
 $$ 
 Suppose that $\ker (\s)\cap \CC_R\neq\emptyset$. Then, by Lemma \ref{a15Sep24}, $\CK (\s (R))<\CK (R)$. By Proposition \ref{A25Sep24}, $\CK (\s (R))=\CK (R)$, a contradiction. Therefore, $\ker (\s)\cap \CC_R\neq \emptyset$, as required. $\Box$

{\bf Licence.} For the purpose of open access, the author has applied a Creative Commons Attribution (CC BY) licence to any Author Accepted Manuscript version arising from this submission.

{\bf Disclosure statement.} No potential conflict of interest was reported by the author.

{\bf Data availability statement.} Data sharing not applicable – no new data generated.

{\bf Funding.} This research received no specific grant from any funding agency in the public, commercial, or not-for-profit sectors.

\small{
}

\begin{minipage}[t]{0.5\textwidth}
	V. V. Bavula
	
	School of Mathematical and Physical Sciences
	
	Division of Pure Mathematics
	
	University of Sheffield
	
	Hicks Building
	
	Sheffield S3 7RH
	
	UK
	
	email: v.bavula@sheffield.ac.uk
\end{minipage}
\begin{minipage}[t]{0.5\textwidth}
	
\end{minipage}


\begin{thebibliography}{99}


\bibitem{Gabriel-1962} P. Gabriel, Des cat\'{e}gories abb\'{e}liennes, {\em Bull. Sot. Math. France} {\bf 90} (1962), 323--448.

\bibitem{Gordon-Robson-1973} R. Gordon and J. C. Robson, Semiprime Rings with Krull Dimension are Goldie, {\em J. Algebra} {\bf 25} (1973) 519-521.

\bibitem{Gordon-Robson-KrullDim} R. Gordon and J. C. Robson, Krull dimension, {\em Amer.  Math. Soc. Mem.},   No. 133 American Mathematical Society, Providence, RI, 1973, ii+78 pp. 

\bibitem{Leroy-Mat-2013} A. Leroy and J.  Matczuk, Ring endomorphisms with large images, {\em  Glasg. Math. J.} {\bf  55} (2013), no. 2, 381--390.

\bibitem{MR} J. C. McConnell and J. C. Robson, Noncommutative Noetherian Rings, Graduate Studies in Mathematics, Vol. 30, 2001.







\end{thebibliography}
\end{document}